\documentclass{amsart}

\usepackage{amsmath, amssymb}
\input xy
\xyoption{all}
\begin{document}

\newtheorem{innercustomthm}{Theorem}
\newenvironment{customthm}[1]
  {\renewcommand\theinnercustomthm{#1}\innercustomthm}
  {\endinnercustomthm}
  
\newtheorem{theorem}{Theorem}[section]
\newtheorem{proposition}[theorem]{Proposition}
\newtheorem{lemma}[theorem]{Lemma}
\newtheorem{corollary}[theorem]{Corollary}
\newtheorem{fact}[theorem]{Fact}

\theoremstyle{definition}
\newtheorem{definition}[theorem]{Definition}
\newtheorem{conjecture}[theorem]{Conjecture}
\newtheorem{notation}[theorem]{Notation}

\theoremstyle{remark}
\newtheorem{remark}[theorem]{Remark}
\newtheorem{example}[theorem]{Example}
\newtheorem{question}[theorem]{Question}

\numberwithin{equation}{section}

\def\id{\operatorname{id}}
\def\height{\operatorname{ht}}
\def\alg{\operatorname{alg}}
\def\Frac{\operatorname{Frac}}
\def\Const{\operatorname{Const}}
\def\spec{\operatorname{Spec}}
\def\span{\operatorname{span}}
\def\exc{\operatorname{Exc}}
\def\Div{\operatorname{Div}}
\def\cl{\operatorname{cl}}
\def\mer{\operatorname{mer}}
\def\trdeg{\operatorname{trdeg}}
\def\ord{\operatorname{ord}}
\def\rank{\operatorname{rank}}
\def\loc{\operatorname{loc}}
\def\kloc{\operatorname{K-loc}}
\def\acl{\operatorname{acl}}
\def\dcl{\operatorname{dcl}}
\def\tp{\operatorname{tp}}
\def\dcf{\operatorname{DCF}_0}
\def\CC{\mathbb C}
\def\C{\mathcal C}
\def\U{\mathcal U}
\def\PP{\mathbb P}


\def\Ind#1#2{#1\setbox0=\hbox{$#1x$}\kern\wd0\hbox to 0pt{\hss$#1\mid$\hss}
\lower.9\ht0\hbox to 0pt{\hss$#1\smile$\hss}\kern\wd0}
\def\ind{\mathop{\mathpalette\Ind{}}}
\def\Notind#1#2{#1\setbox0=\hbox{$#1x$}\kern\wd0\hbox to 0pt{\mathchardef
\nn=12854\hss$#1\nn$\kern1.4\wd0\hss}\hbox to
0pt{\hss$#1\mid$\hss}\lower.9\ht0 \hbox to
0pt{\hss$#1\smile$\hss}\kern\wd0}
\def\nind{\mathop{\mathpalette\Notind{}}}

\title[Model theory and the DME]{Model theory and the DME: a survey}

\author{Rahim Moosa}
\address{Department of Pure Mathematics, University of Waterloo, Ontario, Canada}
\email{rmoosa@uwaterloo.ca}
\thanks{R. Moosa was supported by an NSERC Discover Grant and Discovery Accelerator Supplement.}
\date{\today}

\begin{abstract}
Recent work using the model theory of differentially closed fields to answer questions having to do with the Dixmier-Moeglin equivalence for (noncommutatve) finitely generated noetherian algebras, and for (commutative) finitely generated Poisson algebras, is here surveyed, with an emphasis on the model-theoretic and differential-algebraic-geometric antecedents.
\end{abstract}

\maketitle

\setcounter{tocdepth}{1}
\tableofcontents

\section{Introduction}

In the last five years a small body of work has arisen around a certain application of model theory to algebra.
It appears in the papers~\cite{BLLSM, BLSM, LLS, isolated} authored by, in various combinations, Jason Bell, St\'ephane Launois, Omar Le\'on S\'anchez, and myself.
It is my intention in this article to survey that work, but also to trace its model-theoretic roots and to discuss other related recent developments.

The model theory involved has to do with the structure of finite rank definable sets in differentially closed fields of characteristic zero, and revolves around the stability-theoretic notions of internality and orthogonality.
However, in the context of differentially closed fields, the necessary stability-theoretic ideas have almost entirely algebro-geometric characterisations, and my approach here will be to express things as much as possible in this way.
I hope the survey will thereby be more generally accessible (though I may be trading in stability-theoretic prerequisites for algebro-geometric ones).

The algebra to which this model theory is being applied is somewhat surprising: it arises from the work of Dixmier and Moeglin in the nineteen seventies on the representation theory of noncommutative algebras.
(The ``DME" of the title stands for ``Dixmier-Moeglin equivalence".)
The specific problems are explained in some detail in the last section, where the main results obtained are stated explicitly with proof outlines sketched.
Let me only say now that they have to do with the structure of prime ideals in noncommutative finitely generated noetherian algebras on the one hand, and the structure of prime Poisson ideals in commutative finitely generated Poisson algebras on the other.

I will spend much of this survey narrating a particular line of work in the model theory of differentially closed fields that begins with Hrushovski's interpretation, in the nineteen nineties, of a theorem of Jouanolou on algebraic foliations.
This is a finiteness theorem on codimension one subsets of certain differential-algebraic varieties.
I will give several variations and generalisations of this theorem, and then explain how certain well-studied constructions in differential-algebraic geometry, like that of the Manin kernel, show that the theorem cannot in general be extended to higher codimension.
However, under the additional constraint of ``analysability to the constants", a higher codimension extension of the Jouanolou-type theorems does hold.
In particular, I will explain why it holds for finite rank definable groups over the constants.
These results, both positive and negative, will then be applied to the Dixmier-Moeglin equivalence problem.

Two asides that take us away from differentially closed fields are included; one in the abstract setting of finite rank types in stable theories, and the other on an analogue of these ideas in the {\em difference}-algebraic setting.

\subsection{Acknowledgements}
This survey is based on a pair of talks I gave on the subject; one during the {\em Model theory and applications} workshop at the Institut Henri Poincar\'e in 2018, and the other at the workshop on {\em Interactions between representation theory and model theory} at the University of Kent in 2019.
I thank both institutions, and the programme organisers, for their hospitality.
I also thank Jason Bell for many useful discussions.

\bigskip
\section{Differential-algebraic preliminaries}

\noindent
A certain amount of familiarity with differentially closed fields and differential-algebraic geometry will be necessary, and the preliminaries here are intended to recall the relevant notions.
In particular, we discuss $D$-varieties.
We encourage the reader with even a small amount of exposure to the subject to skip to the next section, and only look back when she comes across something unfamiliar.

All the fields in this paper will be of characteristic zero.

By a {\em derivation} we mean an additive operator that satisfies the Leibniz rule $\delta(xy)=x\delta(y)+\delta(x)y$.

Fix a differential field $(K,\delta)$.
For each $n>0$, the derivation induces on $K^n$ a topology that is finer than the Zariski topology, called the {\em Kolchin topology}.
Its closed sets are the zero sets of {\em differential-polynomials}, that is, expressions of the form $P(x,\delta x,\dots,\delta^\ell x)$ where $x=(x_1,\dots,x_n)$, $\delta^ix=(\delta^ix_1,\dots,\delta^ix_n)$, and $P$ is an ordinary polynomial over $K$ in $(\ell+1)n$ variables.
A {\em differential-rational} function is then a ratio of differential-polynomials.

The geometry of Kolchin closed sets is only made manifest when we work in an existentially closed differential field, that is, a {\em differentially closed} field $(K,\delta)$.
This means that any finite system of differential-polynomial equations and inequations over $K$ that has a solution in some differential field extension of $(K,\delta)$, already has a solution in $(K,\delta)$.
In particular, $K$ is algebraically closed, as is its {\em field of constants} $K^\delta:=\{a\in K:\delta a=0\}$.
The class of differentially closed fields is axiomatisable in the language $\{0,1,+,-,\times,\delta\}$ of differential rings, and the corresponding theory is denoted by~$\dcf$.
It is an $\omega$-stable theory that admits quantifier elimination and the elimination of imaginaries.
It serves as the appropriate theory in which to study differential-algebraic geometry.

The Kolchin topology is noetherian and we have, therefore, the attendant notion of {\em irreducibility}.
Moreover, to an irreducible Kolchin closed set over a differential subfield~$k$, we can associate the {\em generic type} $p(x)$ over $k$ which asserts that $x\in X$ but $x\notin Y$ for every proper Kolchin closed subset $Y\subset X$ over $k$.

The Kolchin closed sets we will be mostly interested in will arise in a very particular way from algebraic varieties equipped with a ``twisted vector field".
Fix a differential subfield $k$ of parameters.
If $V\subseteq K^n$ is an irreducible affine algebraic variety over $k$, then the {\em prolongation} of $V$ is the algebraic variety $\tau V\subseteq K^{2n}$ over~$k$ whose defining equations are
\begin{eqnarray*}
P(x_1,\dots,x_n)&=&0\\
P^\delta(x_1,\dots,x_n)+\sum_{i=1}^n\frac{\partial P}{\partial x_i}(x_1,\dots,x_n)\cdot y_i&=&0
\end{eqnarray*}
for each $P\in I(V)\subseteq k[x_1,\dots,x_n]$.
Here $P^\delta$ denotes the polynomial obtained by applying $\delta$ to the coefficients of $P$.
The projection onto the first $n$ coordinates gives us a surjective morphism $\pi:\tau V\to V$.

Note that if $\delta=0$ on $k$, that is, $V$ is defined over $K^\delta$, then the equations for the prolongation reduce to the equations for the tangent bundle $TV$.
In general, $\tau V$ will be a torsor for $TV$; for each $a\in V$ the fibre $\tau_a V$ is an affine translate of the tangent space $T_aV$.

Taking prolongations is a functor which acts on morphisms by acting on their graphs.
Moreover, the prolongation construction extends to abstract varieties by patching over an affine cover in a natural and canonical way.

The following notion extends that of an algebraic vector field:

\begin{definition}
\label{defn-dvar}
A {\em $D$-variety over $k$} is a pair $(V,s)$ where $V$ is an irreducible algebraic variety over $k$ and $s:V\to\tau V$ is a regular section to the prolongation  defined over $k$.
A {\em $D$-subvariety} of $(V,s)$ is then a $D$-variety $(W,t)$ where $W$ is a closed subvariety of $V$ and $t=s|_W$.
\end{definition}

Let $k[V]$ be the co-ordinate ring of an irreducible affine variety $V$ over $k$.
Then the possible affine $D$-variety structures on $V$ correspond bijectively to the extensions of $\delta$ to a derivation on $k[V]$.
Indeed, given $s:V\to\tau V$, write $s(x)=\big(x,s_1(x),\dots,s_n(x)\big)$ in variables $x=(x_1,\dots,x_n)$.
There is a unique derivation on the polynomial ring $k[x]$ that extends $\delta$ and takes $x_i\to s_i(x)$.
The fact that $s$ maps $V$ to $\tau V$ will imply that this induces a derivation on $k[V]=k[x]/I(V)$.
Conversely suppose we have an extension of $\delta$ to a derivation on $k[V]$, which we will also denote by $\delta$.
Then we can write $\delta\big(x_i+I(V)\big)=s_i(x)+I(V)$ for some polynomials $s_1,\dots,s_n\in k[x]$.
The fact that $\delta$ is a derivation on $k[V]$ extending that on $k$ will imply that $s=(\id,s_1,\dots,s_n)$ is a regular section to $\pi:\tau V\to V$.
It is not hard to verify that these correspondences are inverses of each other.
Moreover, the usual correspondence between subvarieties of $V$ defined over $k$ and prime ideals of $k[V]$, restricts to a correspondence between the $D$-subvarieties of $(V,s)$ defined over $k$ and the prime {\em differential ideals} of $k[V]$, that is, the prime ideals that are closed under the action of $\delta$.

Suppose, now, that $(V,s)$ is a $D$-variety over $k$.
The equations defining the prolongation are such that if $a\in V(K)$ then $\nabla(a):=(a,\delta a)\in\tau V(K)$.
Consider, therefore, the Kolchin closed set
$$(V,s)^\sharp(K):=\{a\in V(K):s(a)=\nabla(a)\}.$$
To say that $s(a)=\nabla(a)$ is to say, writing $s=(\id,s_1,\dots,s_n)$ in an affine chart, that $\delta a_i=s_i(a)$ for all $i=1,\dots,n$.
That is,  $(V,s)^\sharp(K)$ is a Kolchin closed set defined by order $1$ differential-polynomial equations.

The reason that the Kolchin closed sets $(V,s)^\sharp(K)$ are of particular importance is that every finite-dimensional Kolchin closed set is, up to differential-rational bijection, of this form.
Here, a Kolchin closed set $X$ over $k$ is {\em finite-dimensional} if for every $a\in X$ the transcendence degree of $k(a,\delta a,\delta^2a,\dots)$ over $k$ is finite.

Finally, we will be concerned at times with the group objects in this category.
Suppose $(G,s)$ is a $D$-variety where $G$ happens to be an algebraic group.
Note that by the functoriality of $\tau$, and the fact that $\tau(G\times G)=\tau G\times\tau G$, there is an algebraic group structure on $\tau G$ induced by that on $G$.
We say that $(G,s)$ is a {\em $D$-group} if $s:G\to\tau G$ is a homomorphism of algebraic groups.
In that case, $(G,s)^\sharp(K)$ will be a definable group in $(K,\delta)$, and in fact every finite-dimensional definable group will be, up to definable isomorphism, of this form.

\bigskip
\section{Jouanalou-type theorems}
\label{sect-jou}

\noindent
Extending and interpreting a theorem of Jouanalou~\cite{jouanolou} on algebraic foliations, Hrushovski proved, in the unpublished manuscript~\cite{hrushovski-jouanolou} dating from the mid nineteen nineties, the following striking fact in differential-algebraic geometry.

\begin{theorem}[Hrushovski]
\label{udijan}
Working in a saturated model $(K,\delta)$ of $\dcf$ with constant field $\C$, suppose that $X\subseteq K^n$ is an irreducible Kolchin closed subset over a subfield $k\subseteq\C$.
If $X$ admits no nonconstant differential rational functions to $\C$ over $k$ then it has only finitely many differential hypersurfaces over $k$.\footnote{For a published proof of this theorem see~\cite[Theorem~5.7]{freitag-moosa} where it is also generalised to several commuting derivations and to the case when the parameters are possibly nonconstant.}
\end{theorem}

Here, by a {\em differential hypersurface} we mean an irreducible Kolchin closed $Y\subset X$ such that for generic $x\in X$ and $y\in Y$,
$$\trdeg_kk(y,\delta y,\dots,\delta^ty)=\trdeg_kk(x,\delta x,\dots,\delta^tx)-1$$
for all sufficiently large $t$.

The model-theoretic content here, besides the fact that one works in $\dcf$ as the ambient theory in which to do differential-algebraic geometry, is that $X$ admitting no nonconstant differential rational functions to $\C$ over $k$ is equivalent to the generic type of $X$ over $k$ being {\em weakly orthogonal} to the constants: every realisation of this type is independent over $k$ from any finite tuple of constants.

What is required for the intended application to the Dixmier-Moeglin equivalence problem, however, is a slight variant of Hrushovski's theorem.
First, we are only interested in finite-dimensional Kolchin closed sets.
Therefore, we can restrict attention to those $X$ of the form $(V,s)^\sharp(K)$ for some affine $D$-variety $(V,s)$ over~$k$.
The assumption in Hrushovski's theorem of no nonconstant differential rational functions to the constants translates to the $D$-variety being ``$\delta$-rational" in the following terminology of~\cite{BLSM}.

\begin{definition}
\label{deltarational}
An affine $D$-variety $(V,s)$ over $k$ is called {\em $\delta$-rational} if the constant field of the induced derivation on $k(V)$ is $k$.
\end{definition}

The conclusion of Hrushovski's theorem, that of having only finitely many differential hypersurfaces, becomes the statement that $(V,s)$ has only finitely many $D$-subvarieties over $k$ of codimension one, or expressed algebraically, that $(k[V],\delta)$ has only finitely many height one prime differential ideals.
So in this finite-dimensional setting Hrushovski's theorem becomes a rather concrete statement in differential algebra.
On the other hand, the intended application does require the more general context of several (possibly noncommuting) derivations.
Such a version appears in~\cite[Theorem~6.1]{BLLSM} as follows:

\begin{theorem}[Bell, Launois, Le\'on S\'anchez, Moosa]
\label{bllsm}
Suppose $A$ is a finitely generated integral $k$-algebra, where $k$ is an algebraically closed field of characteristic zero, and $\delta_1,\dots,\delta_m$ are $k$-linear derivations on $A$.
If $(A,\delta_1,\dots,\delta_m)$ has infinitely many height one prime differential ideals then there is some $f\in\Frac(A)\setminus k$ with $\delta_i(f)=0$ for all $i=1,\dots,m$.
\end{theorem}

There are now two rather distinct proofs of this in the literature: the one appearing in~\cite{BLLSM} is entirely algebraic while another following the method of Hrushovski (so relying on an extension of Jouanolou's theorem) appeared later in~\cite[Theorem~4.2]{freitag-moosa}, where it is generalised to the infinite-dimensional setting.

It turns out that the phenomena exhibited by the above theorems is not specific to differential-algebraic geometry.
A difference-algebraic analogue was proved in 2010 (independently) by Bell-Rogalski-Sierra~\cite{BRS} and Cantat~\cite{cantat}.
We will say more about that in~$\S$\ref{sect-diff} below.
More recently it has been shown that such finiteness theorems are truly ubiquitous, and a very general unifying version that appears in~\cite{invariant} is as follows:

\begin{theorem}[Bell, Moosa, Topaz]
\label{bmt}
\label{main}
Let $X$ be an algebraic variety, $Z$ an irreducible algebraic scheme of finite type, and $\phi_1,\phi_2:Z\to X$ rational maps whose restrictions to $Z_{\operatorname{red}}$ are dominant, all over an algebraically closed field $k$ of characteristic zero.
Suppose there exist nonempty Zariski open subsets $V\subseteq Z$ and $U\subseteq X$ such that the restrictions $\phi_1^V,\phi^V_2:V\to U$ are dominant regular morphisms, and there exist infinitely many algebraic hypersurfaces $H$ on $U$ satisfying
$$(\phi_1^V)^{-1}(H)=(\phi^V_2)^{-1}(H).$$
Then there exists $g\in k(X)\setminus k$ such that $g\phi_1=g\phi_2$.
\end{theorem}

Given an integral differential $k$-algebra $(A,\delta_1,\dots,\delta_m)$, consider $X=\spec(A)$ and $Z=\spec(R)$ where $R:=A[\epsilon_1,\dots,\epsilon_m]/(\epsilon_1,\dots,\epsilon_m)^2$.
We have $\phi_1:Z\to X$ coming form the natural $k$-algebra inclusion of $A$ in $R$ and $\phi_2:Z\to X$ coming form the $k$-algebra homomorphism given by $a\mapsto a+\delta_1(a)\epsilon_1+\cdots+\delta_m(a)\epsilon_m$.
Applying Theorem~\ref{bmt} to this yields Theorem~\ref{bllsm}.
(The proof of Theorem~\ref{bmt} goes via a reduction to Theorem~\ref{bllsm}, and hence does not constitute a new proof of the latter.)

\bigskip
\section{Counterexamples in higher codimension}
\label{sect-counterex}

\noindent
The Jouanolou-type theorems discussed above assert that under certain conditions (weak orthogonality to the constants, or $\delta$-rationality) there are only finitely many hypersurfaces.
In this section we survey some constructions showing that these theorems cannot in general be extended beyond hypersurfaces.

To begin with, note that for an algebraic $D$-variety to have infinitely many $D$-hypersurfaces is equivalent to having Zariski-dense many $D$-hypersurfaces.
The algebraic counterpart of this equivalence is that there are only finitely many height~$1$ prime differential ideals  if and only if the intersection of all height $1$ prime differential ideals is nontrivial.
It is not hard to see that this equivalence is no longer true in higher codimension.
In thinking about stregthenings of Jouanalou-type theorems, therefore, one would certainly not expect to be able to conclude the existence of only finitely many $D$-subvarieties, but one might speculate that the union of the proper $D$-subvarieties is not Zariski dense.
The following terminology was introduced in~\cite{BLSM}.

\begin{definition}
\label{deltalocallyclosed}
An affine $D$-variety $(V,s)$ over an algebraically closed field $k$ is called {\em $\delta$-locally closed} if the union of all its proper $D$-subvarieties over $k$ is not Zariski dense.
\end{definition}
\begin{remark}
Algebraically, this is equivalent to saying that in the induced differential ring $(k[V],\delta)$ the intersection of all nontrivial prime differential ideals is not trivial.
Model-theoretically,  being $\delta$-locally closed means that the Kolchin generic type of $(V,s)^\sharp$ over $k$ is isolated.
See~\cite[$\S2.4$]{BLSM} for proofs of these equivalences.
\end{remark}

The higher codimension analogues of Theorems~\ref{udijan} and~\ref{bllsm}, at least in one derivation and in the finite-dimensional case, would assert that {\em $\delta$-rational implies $\delta$-locally closed} for an affine $D$-variety over the constants.
The constructions we now exhibit show that this is not the case.
More precisely, they will show the following fact, which should be contrasted with Theorem~\ref{bllsm}.

\begin{fact}
\label{counterex}
In every dimension $d\geq 3$, there are $\delta$-rational but not $\delta$-locally closed affine $D$-varieties over~$k$.
Expressed algebraically, there exists a finitely generated integral $k$-algebra $A$ of Krull dimension $d$ equipped with a derivation $\delta$ such that the constants of $\Frac(A)$ is $k$ but the intersection of the nontrivial prime differential ideals is trivial.
\end{fact}

\medskip
\subsection{Parametrised Manin kernels}
Manin kernels were used by both Buium and Hrushovski in their proofs of the function field Mordell-Lang conjecture.
There are several expositions of this material available, so we will be brief.
See Marker~\cite{marker} and Bertrand-Pillay~\cite{bertrandpillay} for more details on Manin kernels, and~\cite{BLLSM} for more details on how they witness Fact~\ref{counterex}.

Fix an algebraically closed field $k$ and consider the differential field $L=k(t)$ with $\delta=\frac{d}{dt}$.
Let $E$ be an elliptic curve over $L$.
(A similar construction works with abelian varieties.)
The {\em universal vectorial extension of $E$}, denoted by $\widehat E$, is a $2$-dimensional connected commutative algebraic group over $L$ equipped with a surjective morphism of algebraic groups $p:\widehat E\to E$ whose kernel is an algebraic vector group, satisfying the universal property that $p$ factors uniquely through every such extension of $E$ by a vector group.
By functoriality, the prolongation $\tau\widehat E$ inherits the structure of a connected commutative algebraic group in such a way that $\pi:\tau\widehat E\to \widehat E$ is a morphism with kernel the Lie algebra $T_0\widehat E$.
The composition $p\circ \pi:\tau\widehat E\to E$ is thus again an extension of $E$ by a vector group, and so by the universal property there is a unique morphism of algebraic groups $s:\widehat E\to\tau\widehat E$ over $L$ such that $p=p\circ\pi\circ s$.
It follows that $s$ is a section to $\pi$ and so $(\widehat E,s)$ is a $D$-group over $L$.
An important property of this $D$-group is that  $(\widehat E,s)^\sharp(L^{\alg})$ is Zariski-dense in $\widehat E$, which follows from the fact that the torsion on $\widehat E$, like that on $E$ itself, is Zariski-dense.
If we make the further assumption that $E$ is not isomorphic to an elliptic curve over $k$, that is, $E$ does not descend to the constants, then it can be shown that the field of constants of $\big(L(\widehat E),\delta\big)$ is $k$.

At this point we can forget the group structure.
Letting $V$ be an appropriate affine open subset of $\widehat E$, what we have obtained is an affine algebraic $D$-variety $(V,s)$ over $L$ of dimension $2$ for which $(V,s)^\sharp(L^{\alg})$ is Zariski-dense in $V$ while the constant field of the induced derivation $\delta$ on $L(V)$ is $k$.

Now, because $L$ is a one-dimensional function field over $k$ we can witness $V$ as the generic fibre of a morphism $X\to C$ from a three-dimensional irreducible affine algebraic variety onto an affine curve, all defined over the base field $k$.
Moreover, we can arrange it so that the section $s$ on $V$ extends to a $D$-variety structure $s_X:X\to TX$ on the total space.
(Note that as $X$ is defined over the constants the prolongation is just the tangent bundle.)
Since $k(X)=L(V)$, we have that the induced derivation on $k(X)$ has constant field $k$.
So $(X,s_X)$ is $\delta$-rational.
But $X$ has many $D$-curves.
Indeed, each $L^{\alg}$-point of $(V,s)^\sharp$ gives rise to a $D$-subvariety of $X$ over~$k$ that projects finitely onto $C$.
As the set of such points is Zariski-dense on the generic fibre $V$, the union of these $D$-curves is Zariski-dense in $X$.
Note that the $D$-curves are of codimension~$2$ in $X$.
In any case, $(X,s_X)$ is not $\delta$-locally closed and therefore witnesses Fact~\ref{counterex}.

The construction can de adjusted to produce higher Krull dimension examples -- by replacing $L$ with higher dimensional function fields though still equipped with a derivation whose constants are $k$.

\medskip
\subsection{The $j$ function}
The $j$ function is a classically known analytic function on the upper half complex plane that parametrises elliptic curves over~$\mathbb C$.
It satisfies a certain order three algebraic differential equation of the form $\delta^3x=f(x,\delta x,\delta^2 x)$ where $f$ is a rational function over $\mathbb Q$.
Its set of solutions in a model $(K,\delta)\models\dcf$ is therefore, upto generic definable bijection, of the form $(V,s)^\sharp(K)$ for some $D$-variety $(V,s)$ over $k:=\mathbb Q^{\alg}$ with $\dim V=3$.
This Kolchin closed set was studied from the model-theoretic point of view by Freitag and Scanlon in~\cite{freitagscanlon} using Pila's Ax-Lindemann-Weierstrass with derivatives theorem~\cite{pila}.
They show that it is a strongly minimal set with trivial pregeometry that is not $\omega$-categorical; a first such example.
The failure of $\omega$-categoricity is due to Hecke correspondences which ensure that for any generic $a\in (V,s)^\sharp(K)$ there are infinitely many $k(a)^{\alg}$-points in $(V,s)^\sharp$.
By strong minimality this is a Zariski-dense set in $V$.
We are lead to consider, therefore, the $6$-dimensional $D$-variety $(V^2,s^2)$.
The triviality of the pregeometry associated to $(V,s)^\sharp(K)$ ensures that $\big(k(V^2),\delta\big)$ has~$k$ as its field of constants.
That is, $(V^2,s^2)$ is $\delta$-rational.
Now, consider the first co-ordinate projection $V^2\to V$.
Fixing generic $a\in (V,s)^\sharp(K)$, each $k(a)^{\alg}$-point of $(V,s)^\sharp$, viewed as living on the fibre of $V^2\to V$ over $a$, gives rise to a $D$-subvariety of $V^2$ over~$k$ that projects finitely onto $V$.
As the set of such points is Zariski-dense on the fibre above $a$, the union of these $D$-subvarieties is Zariski-dense in $V^2$.
Notice that these $D$-subvarieties are of dimension (and codimension) $3$.
In any case, this example witnesses again Fact~\ref{counterex}.

\bigskip
\section{Analysability in the constants and $D$-groups}
\label{dgroupsect}

\noindent
The counterexamples of the previous section notwithstanding, it is reasonable to ask for which $D$-varieties do higher codimension strengthenings of the Jouanolou-type theorems hold.
One theorem along these lines from~\cite{BLSM} is:

\begin{theorem}[Bell, Le\'on S\'anchez, Moosa]
\label{dgroupdme}
For $D$-subvarieties of affine $D$-groups over the constants, $\delta$-rational does imply  $\delta$-locally closed.
\end{theorem}

Theorem~\ref{dgroupdme} is a consequence of combining known structural results about finite rank groups definable in $\dcf$ with a more general sufficient condition on $D$-varieties for when $\delta$-rationality implies $\delta$-local-closedness.
This latter condition, in the terminology of~\cite{BLSM}, is ``compound isotriviality".

\begin{definition}
Fix a differential field $(k,\delta)$ and a $D$-variety $(V,s)$ over $k$.

We say that $(V,s)$ is {\em isotrivial} if there is some differential field extension $F\supseteq k$ such that $(V,s)$ is $D$-birationally equivalent over $F$ to a $D$-variety of the form $(W,0)$ where $W$ is defined over the constants and $0$ is the zero section.

We say that $(V,s)$ is {\em compound isotrivial} if there exists a sequence of $D$-varieties $(V_i,s_i)$ over $k$, for $i=0,\dots,\ell$, with dominant $D$-rational maps over $k$
$$
\xymatrix{
V=V_0\ar[r]^{ \ \ \ f_0}&V_1\ar[r]^{f_1}&\cdots\ar[r]&V_{\ell-1}\ar[r]^{f_{\ell-1} \ \ \ }&V_\ell=0}
$$
such that the generic fibres of each $f_i$ are isotrivial.
We say in this case that $(V,s)$ is compound isotrivial in {\em $\ell$ steps}.
\end{definition}

\begin{remark}
These notions have model-theoretic characterisations, which, indeed, are where they come from.
The model-theoretic meaning of isotriviality is that
that the Kolchin generic type of $(V,s)^\sharp$ over $k$ is {\em internal to the constants}.
Compound isotriviality says that generic type is {\em analysable in the constants}.
We point the reader to~\cite{gst} for the definitions of internality and analysability in stable theories.
There is also, of course, an algebraic characterisation in terms of the differential structure on $k[V]$, but it is not very illuminating and we leave it to the reader to work out if he desires.
\end{remark}

Stated informally, compound isotriviality is saying that the $D$-variety is built up via a finite sequence of fibrations by $D$-varieties that at least after base change are trivial vector fields.
This is a much more rich class of $D$-varieties than one may at first glance expect, as witnessed, for example, by the constructions of Jin in~\cite{jin}.

In any case, for such $D$-varieties we do have the desired strengthening of the Jouanalou-type theorems; Proposition~2.13 of~\cite{BLSM} says that {\em for a compound isotrivial $D$-variety, $\delta$-rationality implies $\delta$-local-closedness}.
Model-theoretically, this says that for types analysable in the constants, being weakly orthogonal to the constants implies isolation.
We do not expose the proof here, but it may be worth pointing out to the model-theorist why it is true in at least the base case when the $D$-variety is isotriviality: If a stationary type is internal to the constants then by $\omega$-stability we have a definable {\em binding group} action which must be transitive if the type is in addition weakly orthogonal to the constants.
As a result, such types define an orbit of a definable group action and hence are isolated.

Now, it is well known that $D$-groups over the constants, and their $D$-subvarieties, are compound isotrivial.
Indeed, given $(G,s)$ with $Z$ the center of $G$ and
$$H=\{g\in G:s(g)=0\},$$
one considers the normal sequence of $D$-subgroups
$$0\leq Z\cap H \leq Z \leq G.$$
Since the section is zero on $H$ by definition, $Z\cap H$ is isotrivial.
That $Z/(Z\cap H)$ is isotrivial follows from the theory of the logarithmic derivative homomorphism on a connected commutative algebraic group over the constants, see~\cite{marker}.
Finally, that $G$ modulo its center is isotrivial is a theorem of Buium~\cite{buium} and Pillay-Kowalski~\cite{kowalskipillay}.
We therefore have the $3$-step analysis given by the sequence of fibrations
$$G\rightarrow G/(Z\cap H)\rightarrow G/Z\rightarrow 0.$$
This will restrict to an analysis of any $D$-subvariety of $(G,s)$.
Theorem~\ref{dgroupdme} follows.

\bigskip
\section{An abstract resolution?}

\noindent
As we have noted along the way, strengthening the Jouanolou-type theorems to all codimensions requires proving that $\delta$-rational implies $\delta$-locally closed.
We know this is not the case for all $D$-varietes.
Instead of restricting our attention to certain $D$-varieties -- as we did in the previous section -- here we will describe work that finds the right condition to replace $\delta$-rationality with in order to get a characterisation of $\delta$-local-closedness.
While everything can be translated into differential algebra and differential-algebraic geometry, the condition we are after is best expressed in the language of geometric stability theory, with which we will therefore assume familiarity in this section.

Working in a saturated model $(K,\delta)\models\dcf$ with constant field $K^\delta=\C$, let us fix a type $p=\tp(a/A)$ of finite rank.
Up to definable bijection these are precisely the Kolchin generic types associated to $D$-varieties.
Local-closedness of the $D$-variety is equivalent to $p$ being isolated.
In~\cite{isolated} the following characterisation of isolation was established.

\begin{theorem}[Le\'on S\'anchez, Moosa]
\label{the1dcf}
Suppose $p=\tp(a/A)$ is of finite rank.
Then $p$ is isolated if and only if the following hold:
\begin{itemize}
\item[(i)]
$\displaystyle a\ind_A\C$; and
\item[(ii)]
$\displaystyle a\ind_{Ab}G^\sharp$ for every $b\in\acl(Aa)$ and $G$ a simple abelian variety over $\acl(Ab)$ that does not descend to the constants, where $G^\sharp$ denote the Kolchin closure of torsion in $G$; and
\item[(iii)]
 $\displaystyle a\ind_{Ab}q(K)$ for every $b\in \acl(Aa)$ and $q$ a nonisolated $U$-rank $1$ type over $Ab$ with trivial associated pregeometry.
\end{itemize}
\end{theorem}

Maybe a few words of explanation are in order.
First, condition~(i) is precisely weak $\C$-orthogonality of $p$ -- which, recall, captures the $\delta$-rationality of the corresponding $D$-variety.
So this theorem adds further conditions to $\delta$-rationality in order to characterise $\delta$-closedness (i.e., isolation).
Moreover, those further conditions are in the same spirit as $\delta$-rationality; conditions~(ii) and~(iii) both say that the ``fibrations" of $p$ must be weakly orthogonal to certain other (type) definable sets.
While condition~(ii) is rather concrete and in practice verifiable, condition~(iii) remains more elusive with no complete characterisation of trivial minimal types in $\dcf$ as yet known.

We have referred above to $\tp(a/Ab)$ when $b\in\acl(Aa)$ as a ``fibration" of $\tp(a/A)$.
This terminology comes from considering the special case when $b\in\dcl(Aa)$.
In that case one does obtain a rational map on the $D$-variety corresponding to $p$ and $\tp(a/Ab)$ is the Kolchin generic type of the generic fibre of that map.
As it turns out, however, there are examples showing that it does not suffice to consider only $b\in\dcl(Aa)$, the algebraic closure is necessary.

Theorem~\ref{the1dcf} itself follows from a more general characterisation of isolated finite rank types:
working in a saturated model $\U$ of an arbitrary totally transcendental theory satisfying an additional technical condition\footnote{That every complete non-locally-modular minimal type is nonorthogonal to a nonisolated minimal type over the empty set.},
a finite rank type $p=\tp(a/A)$ is isolated if and only if 
\begin{itemize}
\item[($\dagger$)]
for all $b\in \acl(Aa)$ and $q\in S(Ab)$ nonisolated and minimal, $\displaystyle a\ind_{Ab}q(\U)$.
\end{itemize}
Conditions~(i)--(iii) of Theorem~\ref{the1dcf} is how~($\dagger$) is manifested in $\dcf$ due to the form the Zilber trichotomy takes in that theory.
One obtains also analogous results for $\operatorname{CCM}$, the theory of compact complex manifolds in the langauge of analytic sets.

\bigskip
\section{An aside on the difference case}
\label{sect-diff}

\noindent
We are focusing in this paper on differential algebra, but we would like to briefly mention here an analogue of the Jouanalou-type theorems and their possible higher codimension strengthenings in the {\em difference} setting.
That is, instead of $D$-varieties we consider a {\em $\sigma$-variety}; an irreducible quasi-projective algebraic variety $V$ over an algebraically closed field $k$ equipped with a regular automorphism\footnote{It is is possible to work with dominant rational self-maps instead but the statements would have to be modified somewhat, and we stick to this context for the sake of economy of exposition.} also over~$k$, $\phi:V\to V$ .
Then $\phi$ induces a $k$-linear automorphism of the rational function field, $\sigma:k(V)\to k(V)$.
Let us say that $(V,\phi)$ is {\em $\sigma$-rational} if the fixed field of $\big(k(V),\sigma\big)$, namely the set of $f\in k(V)$ such that $\sigma(f)=f$, is just $k$ itself.
(This is often expressed by saying that $\phi$ {\em does not preserve any nonconstant fibration} on $V$ over~$k$.)
The analogy being drawn here is between derivations with their constants and automorphisms with their fixed points.

We mentioned toward the end of~$\S$\ref{sect-jou} that there is a version of the Jouanolou theorem in this context, though it was discovered much later.
The role of a differential (or $D$-) subvarieties is played now by that of a {\em $\phi$-invariant subvariety}; namely a subvariety $W\subset V$ such that $\phi(W)\subseteq W$.
There is a subtle variance here from the differential case: the subvariety $W$ need not be irreducible, nor need it break up into a union of irreducible $\phi$-invariant subvarieties.
For example, if a point $a\in V(k)$ is periodic but not fixed then it has finite but nontrivial orbit under~$\phi$, and that orbit is a $\phi$-invariant subvariety that is not the union of irreducible $\phi$-invariant subvarieties.
The following Jouanolou-type theorem is a special case of Theorem~\ref{bmt} that appeared originally in~\cite{BRS} and~\cite{cantat}, independently.

\begin{theorem}[Bell, Rogalski, Sierra; Cantat]
\label{diffjou}
Suppose $(V,\phi)$ is a $\sigma$-variety over~$k$.
If $(V,\phi)$ is $\sigma$-rational then there are only finitely many codimension one $\phi$-invariant subvarieties over $k$ on $V$.
\end{theorem}

What about higher codimensions?
That is, if, inspired by the differential case, we were to call $(V,\phi)$ {\em $\sigma$-locally closed} if the union of all proper $\phi$-invariant subvarieties over $k$ is not Zariski dense in $V$; it would be natural to ask:

\begin{question}
\label{diffdme}
Does $\sigma$-rational imply $\sigma$-locally closed?
\end{question}

There is an interesting connection here with a conjecture in arithmetic dynamics.
Note that if $(V,\phi)$ is $\sigma$-locally closed, then any point in $V(k)$ that is outside the Zariski closure of the union of all proper $\phi$-invariant subvarieties over $k$ will have a Zariski-dense orbit under $\phi$.
Indeed, the Zariski closure of that orbit will be $\phi$-invariant and hence cannot be proper.
So, an affirmative answer to Question~\ref{diffdme} would imply the following conjecture from~\cite{medvedev-scanlon}.

\begin{conjecture}[Medvedev, Scanlon]
\label{medvedev-scanlon}
Suppose $(V,\phi)$ is a $\sigma$-variety over~$k$.
If $(V,\phi)$ is $\sigma$-rational then there exists a $k$-point with Zariski-dense orbit.
\end{conjecture}

When $k$ is uncountable the conjecture is known to be true by an older theorem of Amerik and Campana~\cite{amerik-campana} -- so the conjecture is really about countable algebraically closed fields, where it remains open.

Question~\ref{diffdme}, however, is {\em not} open, even for uncountable $k$: as explained in~\cite[Theorem~8.8]{BRS}, work of Jordan~\cite{jordan} on Henon automorphisms $\phi:\mathbb A^2\to\mathbb A^2$ of the affine plane over $\mathbb C$ shows that they have no invariant curves but infinitely many periodic points.
Since the level sets of a nonconstant rational function on $\mathbb A^2$ that is fixed by $\sigma$ would be invariant curves, it follows that there are no such rational functions and $(\mathbb A^2,\phi)$ is $\sigma$-rational.
On the other hand, the orbit of a periodic point is a (finite) $\phi$-invariant subvariety, but the union of these orbits as you range through the infinitely many periodic points is Zariski-dense as its Zariski closure would otherwise be an invariant curve.
Hence $(\mathbb A^2,\phi)$ is not $\sigma$-locally closed.

The obstacle here seems to be the nature of the induced action of $\phi$ on the divisors of $V$ modulo numerical equivalence.
This is a finitely generated free abelian group, and the action is said to be {\em quasi-unipotent} if all of its eigenvalues are roots of unity. 
In~\cite{keeler} it is shown that quasi-unipotence is equivalent to the existence of a ``$\phi$-ample divisor on $X$", which was the condition isolated in~\cite{BRS} as relevant to higher codimension analogues of Theorem~\ref{diffjou}.
Conjecture~8.5 of~\cite{BRS} can therefore be expressed as follows:

\begin{conjecture}[Bell, Rogalski, Sierra]
Suppose $(V,\phi)$ is a projective $\sigma$-variety over $k$ for which the action of $\phi$ on the numerical equivalence classes of divisors is quasi-unipotent.
If $(V,\phi)$ is $\sigma$-rational then it is $\sigma$-locally closed.
\end{conjecture}

They prove a number of cases including when $\dim V\leq 2$ or $\phi$ is part of the action of an algebraic group on $V$.

\bigskip
\section{The Dixmier-Moeglin Equivalence}

\noindent
We will now describe how the narrative we have been pursuing about $\delta$-rationality and $\delta$-local-closedness of $D$-varieties has been applied in recent years to answer some long standing questions in noncommutative and Poisson algebra.
We begin by introducing the Dixmier-Moeglin equivalence in first the classical noncommutative setting, and then its (commutative) Poisson analogue.

\medskip
\subsection{The classical Dixmier-Moeglin equivalence}
Fix an uncountable algebraically closed field $k$ of characteristic zero, and a finitely generated noetherian associative $k$-algebra, $A$, not necesarily commutative.
Suppose $P\subseteq A$ is a prime ideal.
We consider three well-studied conditions on $P$.

The first is algebraic.
Recall that the Goldie ring of fractions, $\Frac(A/P)$, is the localistion of $A/P$ at the set of all regular elements.
It is a simple artinian ring, and hence a matrix algebra over a division ring.
The center of $\Frac(A/P)$ is therefore the center of that division ring, and so a field extension of $k$.
We say that $P$ is {\em rational} if the center of $\Frac(A/P)$ is $k$ itself.\footnote{Here we are assuming $k$ is algebraically closed. In general, {\em rational} would just ask that the center be an algebraic extension of $k$.}

The second is representation-theoretic: we say that $P$ is {\em primitive} if it annhilates a simple left $A$-module.

The final condition is Zariski-topological.
As in the commutative case, the set of prime ideals of $A$, $\spec(A)$, is endowed with a Zariski topology.
We say that $P$ is {\em locally closed} if $\{P\}$ is a locally closed subset of $\spec(A)$.
Equivalently, the intersection of all prime ideals properly containing $P$ is not equal to $P$.

Note that if $A$ were commutative then these three conditions on $P$ would be equivalent; they would all say that $P$ is a maximal ideal.
In general, the algebra $A$ is said to satisfy the {\em Dixmier-Moeglin equivalence} (DME) if a prime ideal of $A$ is rational if and only if it is  primitive if and only if it is locally closed.
This terminology, and the question of which algebras satisfy the DME, arises out of the work of Dixmier and Moeglin in the early nineteen seventies on the representations of universal enveloping algebras of finite-dimensional Lie algebras (which they showed do satisfy the DME).
Following their work it has been shown that locally closed implies primitive implies rational (at least in our context of $A$ being finitely generated and noetherian and $k$ being uncountable).
So the DME problem is the question of whether rational implies locally closed.
(At this point the reader will see at least a linguistic connection to the subject matter of the previous sections.)
Several counterexample were exhibited in the late nineteen seventies by Irving and Lorenz, while for a number of other classes of algebras the DME has been established in the intervening decades.

\medskip
\subsection{The Poisson Dixmier-Moeglin equivalence}
It is often the case that there is an analogy between noncommutative rings and commutative rings equipped with additional (often differential) structure.
One such structure, arising from deformation quantization, is that of an {\em affine Poisson algebra}: a finitely generated commutative integral $k$-algebra $A$ equipped with a Lie bracket $\{,\}$ that is also a biderivation, that is, for each $x\in A$ the operators $\{x,\_\}$ and $\{\_,x\}$ are $k$-linear derivations on $A$.
By a {\em Poisson ideal} is meant an ideal $I\subseteq A$ such that $\{a,x\}\in I$ for all $a\in I$ and $x\in A$.
(The same then holds for $\{x,a\}$ by skew-symmetry.)
Let us fix a Poisson prime ideal $P$ in an affine Poisson algebra $A$, and consider the Poisson analogue of the conditions considered above.

The Poisson structure on $A$ induces a canonical Poisson structure on $A/P$ which then extends uniquely to the fraction field $\Frac(A/P)$.
By the Poisson-center of $\Frac(A/P)$ is meant the subfield $\{f\in \Frac(A/P):\{f,\_\}=0\}$, which is an extension of $k$.
The Poisson prime $P$ is said to be {\em Poisson rational} if the Poisson-center of $\Frac(A/P)$ is just $k$ itself.

We say that $P$ is {\em Poisson primitive} if for some maximal ideal $\mathfrak m$, $P$ is the largest Poisson ideal contained in $\mathfrak m$.\footnote{A more direct analogue with primitivity in the classical case would ask instead for $P$ to be the annihilator of a simple Poisson $A$-module.
The relationship between this and Poisson primitivity  is fruitfully clarified in recent work of Launois and Topley~\cite{LT}.}

Finally, $P$ is {\em Poisson locally closed} if the intersection of all Poisson prime ideals that properly contain $P$ is not equal to $P$.

An affine Poisson algebra $A$ is said to satisfy the {\em Poisson Dixmier-Moeglin equivalence} (PDME) if, for every Poisson prime ideal, these three notions coincide.
As in the classical case, it is known that Poisson locally closed implies Poisson primitive which implies Poisson rational; so that the question is really about Poisson rationality implying Poisson local-closedness.
Unlike the classical case, no counterexamples were found (prior to the work being surveyed here).
The PDME was established in several important contexts over the last twenty years, and Brown and Gordon~\cite{BrownGordon} asked explicitly whether it holds for all affine Poisson algebras.

\medskip
\subsection{Negative results}
The constructions of~$\S$\ref{sect-counterex} yield new negative results for both the DME and the PDME, which we now describe.

The counterexamples of Irving and Lorenz to the classical DME were all of infinite {\em Gelfand-Kirrilov} (GK) {\em dimension}.
This is a noncommutative analogue to Krull dimension in the sense that it agrees with Krull dimension on finitely generated commutative $k$-algebras.
For a finitely generated (noncommutative) $k$-algebra $A$ the GK-dimension can be defined as follows: choose a finite-dimensional $k$-subspace $V\subseteq A$ that contains~$1$ and a set of generators for $A$, and let $V^n$ denote the $k$-span of the set of $n$-fold products of elements of $V$.
Then $\displaystyle \operatorname{GKdim}(A):=\limsup_{n\to\infty}\frac{\log(\dim(V^n))}{\log n}$.
Intuitively, $\displaystyle \operatorname{GKdim}(A)=d<\infty$ means that $\dim(V^n)$ grows like $Cn^d$ for some positive constant $C$.
The GK-dimension does not depend on the choice of $V$.

The first finite GK-dimension algebras not satisfying the DME were found in~\cite{BLLSM}:

\begin{theorem}[Bell, Launois, Le\'on S\'anchez, Moosa]
For each integer $d\geq 4$ there exists a finitely generated noetherian associative (noncommutative) $k$-algebra of GK-dimension $d$ that does not satisfy the DME. 
\end{theorem}

\begin{proof}[Proof outline]
By Fact~\ref{counterex} there is a Krull dimension $d-1$ commutative finitely generated integral $k$-algebra $R$ equipped with a derivation $\delta$ such that the $\delta$-constants of $\Frac(R)$ is $k$ while the intersection of all nontrivial prime differential ideals in $R$ is trivial.
Let $A$ be the one-variable skew-polynomial ring $R[x;\delta]$ where $xr=rx+\delta(r)$ for all $r\in R$.
Then $\operatorname{GKdim}(A)=d$.
The fact that the $\delta$-constants of $\Frac(R)$ is~$k$ implies that the center of the Goldie ring of fractions of $A$ is $k$.
That is, $(0)$ is a rational prime ideal of $A$.
On the other hand, it can be shown that if $P\subset R$ is a prime differential ideal then $PA$ is a prime ideal of $A$.
Moreover, one observes that because the intersection of the nontrivial such $P$s are trivial, the intersection of all such $PA$ is trivial.
So $(0)$ is not locally closed.
\end{proof}

Similar methods give, also in~\cite{BLLSM}, a negative answer to Brown and Gordon's question on the PDME:

\begin{theorem}[Bell, Launois, Le\'on S\'anchez, Moosa]
\label{counterexpdme}
For each $d\geq 4$ there exists a (commutative) affine Poisson algebra of Krull dimension $d$ that does not satisfy the PDME. 
\end{theorem}

\begin{proof}[Proof outline]
Again, we start with the Krull dimension $d-1$ commutative finitely generated integral differential $k$-algebra $R$ given to us by Fact~\ref{counterex}.
Let $A$ be the commutative one-variable polynomial ring $R[x]$ equipped with the Poisson bracket $\{p(x),q(x)\}:=p^\delta(x)q'(x)-p'(x)q^\delta(x)$.
Here  $p^\delta(x)$ again denotes the polynomial obtained by applying $\delta$ to the coefficients of $p$, and $p'(x)$ is the formal derivative of $p$ with respect to $x$.
On shows that the Poisson-center of $\Frac(A)$ is the field of $\delta$-constants of $\Frac(R)$, which in this case is $k$.
That is, $(0)$ is Poisson rational.
Also, if $P\subset R$ is a prime differential ideal then $PA$ is a prime Poisson ideal, and again the intersection of all nontrivial such will be trivial.
Hence $(0)$ is not Poisson locally closed.
\end{proof}

\medskip
\subsection{A corrected PDME}
The model-theoretic and differential-algebraic geometric methods also provide some positive results.
First, the Jouanolou-type finiteness theorems described in~$\S$\ref{sect-jou} suggest a correction to the PDME that appears in~\cite{BLLSM}.
The key is to replace Poisson locally closed by a codimension~1 weakening:

\begin{theorem}[Bell, Launois, Le\'on S\'anchez, Moosa]
Suppose $k$ is uncountable and $A$ is a (commutative) affine Poisson $k$-algebra.
For a Poisson prime ideal $P$ of $A$ the following are equivalent:
\begin{enumerate}
\item
$P$ is Poisson rational,
\item
$P$ is Poisson primitive, and
\item
the set of Poisson prime ideals $Q\supset P$ with $\height(Q)=\height(P)+1$ is finite.
\end{enumerate}
\end{theorem}

\begin{proof}[Proof outline]
The new contributions here are really that~(1) implies~(2) and~(3).
The proof that Poisson rational implies Poisson primitive is differential-algebraic but not model-theoretic and outside the main focus of the present survey.
(It is also the part that uses the uncountability of $k$.)
So we only describe how~(1) implies~(3) is obtained from the Jouanolou-type theorems.

Taking quotients we may assume that $P=(0)$ is Poisson rational.
Fixing generators $x_1,\dots,x_n$ for $A$ we consider the $k$-linear derivations $\delta_i:=\{\_,x_i\}$ on $A$.
Poisson rationality of $(0)$ implies that the field of (total) constants of $(\Frac(A),\delta_1,\dots,\delta_n)$ is~$k$.
Hence, Theorem~\ref{bllsm} implies that $(A,\delta_1,\dots,\delta_n)$ has only finitely many prime differential ideals of height $1$.
But the differential ideals of $(A,\delta_1,\dots,\delta_n)$ are precisely the Poisson ideals of $A$.
So $(0)$ satisfies condition~(3).
\end{proof}

In particular, the PDME does hold for affine Poisson algebras of Krull dimension at most $3$, and so the counterexamples of Theorem~\ref{counterexpdme} are smallest possible.

\medskip
\subsection{Positive results in the Hopf setting}
In~$\S$\ref{dgroupsect} we showed how strengthenings of the Jouanolou-type theorems to higher codimension were possible in the setting of algebraic $D$-groups.
One might expect that the methods initiated in~\cite{BLLSM} could therefore be used to translate these results into positive contributions to both the classical and Poisson DME in the context of Hopf algebras.
This was done in a pair of papers afterwards, whose main accomplishments we now describe.

We do not define Hopf algebras here, except to recall that it is additional structure on a $k$-algebra $A$ in the form of a coproduct $\Delta:A\to A\otimes_k A$, a counit $\epsilon: A\to k$, and an antipode $S:A\to A$, satisfying various properties.
In the finitely generated commutative setting, they are precisely the structure induced on $A$ from an algebraic group structure on $\spec(A)$.
The following theorem, which appeared in~\cite{BLSM}, says that the classical DME holds for Hopf Ore extensions.
It can be seen as evidence for the conjecture of Bell and Leung~\cite{BellLeung} that the DME holds for all Hopf algebras of finite GK-dimension.

\begin{theorem}[Bell, Le\'on S\'anchez, Moosa]
Suppose $R$ is a (commutative) finitely generated integral Hopf $k$-algebra, $\sigma$ is a $k$-linear automorphism of $R$, and $\delta$ is a $k$-linear $\sigma$-derivation on $R$ -- that is $\delta(rs)=\sigma(r)\delta(s)+\delta(r)s$.
Let $A=R[x;\sigma,\delta]$ be the twisted polynomial ring where $xr=\sigma(r)x+\delta(r)$ for all $r\in R$.
If $A$ admits a Hopf algebra structure extending that on $R$ then $A$ satisfies the DME.
\end{theorem}

\begin{proof}[Proof outline]
Let us restrict our attention to the case when $\sigma=\id$ so that $\delta$ is a derivation on $R$ and $A=R[x;\delta]$ is the twisted polynomial ring we were considering before.
Let us also make the simplifying assumption that $(R,\delta)$ is a {\em differential-Hopf} algebra in the sense that $\delta$ commutes with the coproduct -- where the action of $\delta$ on $R\otimes_kR$ is given by $\delta(r\otimes s)=\delta r\otimes s+r\otimes\delta s$.
In this case $R=k[G]$ for some affine algebraic group $G$ and $\delta$ is the derivation induced by a $D$-group structure $s:G\to TG$.
Now suppose $P$ is a rational prime ideal of $A$.
Then $I=P\cap R$ is a prime differential ideal of $R$ and so $V=V(I)$ is a $D$-subvariety of $(G,s)$.
One shows that the rationality of $P$ implies the $\delta$-rationality of $V$, and hence, by Theorem~\ref{dgroupdme}, $V$ is $\delta$-locally closed.
Finally, one argues that this forces $P$ to be locally closed in~$A$.

In general, we cannot assume that $\delta$ commutes with the coproduct on $R$.
However, the assumption that $A=R[x;\delta]$ admits a Hopf algebra structure extending that on $R$ does force $\delta$ to ``almost commute" with the coproduct: for some $a\in R$,
if $\Delta(r)=\sum_jr_{j,1}\otimes r_{j,2}$ then 
$\Delta(\delta r)=\sum_j\delta r_{j,1}\otimes r_{j,2}+ar_{j,1}\otimes \delta r_{j,2}$.
If $a$ were~$1$ this would be the commuting of $\delta$ with $\Delta$.
But in general we can only ensure that $a$ is group-like in the sense that $\Delta(a)=a\otimes a$ (equivalently, $a:G\to\mathbb G_{\text m}$ is a group homomorphism).
That is, $\delta$ is an {\em $a$-coderivation} on $R$.
Geometrically this means that that $(G,s)$ is not necessarily a $D$-group, but only an ``$a$-twisted $D$-group''.
The extension of Theorem~\ref{dgroupdme} from $D$-groups to $a$-twisted $D$-groups is one of the technically challenging aspects of~\cite{BLSM}, but the approach is the same: one shows that every $D$-subvariety of an $a$-twisted $D$-group over the constants is compound isotrivial (though now in at most five, rather than three, steps) and hence $\delta$-rationality implies $\delta$-local-closedness.
\end{proof}

What about in the Poisson setting?
If we assume some added Hopf structure can we prove the DME?

\begin{question}
Suppose $A$ is an affine (commutative) Poisson $k$-algebra equipped with a Hopf algebra structure where the coproduct commutes with the Poisson bracket.\footnote{Here for $\Delta$ to commute with the Poisson bracket means that $\Delta(\{a,b\})=\{\Delta a,\Delta b\}$ where $A\otimes A$ is endowed with the Poisson bracket $\{a_1\otimes b_1,a_2\otimes b_2\}:=\{a_1,a_2\}\otimes b_1b_2+a_1a_2\otimes\{b_1,b_2\}$.}
Does $A$ satisfy the PDME?
\end{question}

This remains an open and intriguing question in general, the state of the art being the following theorem from~\cite{LLS}:

\begin{theorem}[Launois, Le\'on S\'anchez]
The answer to the above question is ``yes" if we assume in addition that the coproduct is cocommutative.
\end{theorem}

\begin{proof}[Proof outline]
We have that $A=k[G]$ for some affine algebraic group $G$.
The assumption that the coproduct on $A$ is cocommutative means that $G$ is commutative.
There are two main steps in the proof, and both use the commutativity of~$G$.

The first step is to extend Theorem~\ref{dgroupdme} to the context of several (possibly noncommuting) derivations, $\mathcal D$.
Much of the formalism of $D$-varieties and $D$-groups goes through here, giving rise to $\mathcal D$-varieties and $\mathcal D$-groups,  except that compound isotriviality for $\mathcal D$-groups (over the total constants) is only established when the underlying algebraic group is commutative.

The next step is to find a finite set of $k$-linear derivations $\mathcal D$ on $A$ that has the same span as the derivations $\{\_,a\}$ for $a\in A$, and such that each derivation in $\mathcal D$ commutes with the coproduct on $A$.
Finding such a $\mathcal D$ in this case is relatively straightforward because of the structure of $G$; being an affine commutative algebraic group it is isomorphic to a product of additive and multiplicative tori.
Whether such $\mathcal D$ exists in general remains open.
In any case, such $\mathcal D$ is then shown to induce a $\mathcal D$-group structure on $G$.
The PDME for $A$ then follows from the extension of Theorem~\ref{dgroupdme} to $\mathcal D$-groups.
\end{proof}
\vfill
\pagebreak


\end{document}